\documentclass{amsart}
\usepackage{amsmath,amssymb,amsthm}
\usepackage{color} 
\definecolor{darkblue}{rgb}{0,0,0.6}
\usepackage[breaklinks, pdftex, ocgcolorlinks,colorlinks=true, citecolor=darkblue, filecolor=darkblue, linkcolor=darkblue, urlcolor=darkblue]{hyperref}
\usepackage[capitalize,nameinlink]{cleveref}
\usepackage{tikz}
\usetikzlibrary{decorations.pathreplacing}

\author{Daniel Kasprowski}
\address{Rheinische Friedrich-Wilhelms-Universit\"at Bonn, Mathematisches Institut,\newline\indent Endenicher Allee 60, 53115 Bonn, Germany}
\email{kasprowski@uni-bonn.de}
\title{Coarse embeddings into products of trees}
\date{\today}

\newtheorem{thm}[equation]{Theorem}
\newtheorem{lem}[equation]{Lemma}
\theoremstyle{definition}
\newtheorem{defi}[equation]{Definition}
\newtheorem{rem}[equation]{Remark}

\newcommand{\cU}{\mathcal{U}}
\newcommand{\cV}{\mathcal{V}}

\newcommand{\R}{\mathbb{R}}
\newcommand{\N}{\mathbb{N}}
\newcommand{\wt}{\widetilde}
\DeclareMathOperator{\diam}{diam}

\renewcommand{\phi}{\varphi}


\begin{document}
	\begin{abstract}
		We give a short and elementary proof of the fact that every metric space of finite asymptotic dimension can be coarsely embedded into a finite product of trees.
	\end{abstract}
	\maketitle
	
	The goal of this note is to prove the following theorem.
		\begin{thm}
		\label{thm:main}
		Let $X$ be a metric space with asymptotic dimension at most $n$. Then there exists a coarse embedding of $X$ into a product of $n+1$ trees. If $X$ is proper, the trees can be chosen to be locally finite.
	\end{thm}
This theorem was proved by Dranishnikov for geodesic metric spaces of bounded geometry \cite[Theorem~3]{dranishnikov_trees}. Dranishnikov and Zarichnyi \cite[Theorem~3.5]{dranishnikov_universal} showed that proper metric spaces can be coarsely embedded into a product of $n+1$ binary trees. Using an ultralimit construction, Guentner, Tessera and Yu \cite[Proof of Theorem~4.1]{GTY_FDC} used this to show that every metric space with asymptotic dimension at most $n$ can be embedded into a product of (n+1)-many 0-hyperbolic spaces. This was used to show that spaces with finite asymptotic dimension have finite decomposition complexity.

The proof presented here is short and elementary. The idea on how the trees are constructed is based on the original proof of Dranishnikov.

We begin by recalling the notions from \cref{thm:main}.
\begin{defi}
	A metric space $X$ is called \emph{proper} if every bounded, closed subspace is compact.
\end{defi}
\begin{defi}
	A metric space $X$ has \emph{asymptotic dimension at most $n$} if for every $r>0$ there exist $f(r)\in\R$ and families $\cU^0,\ldots,\cU^n$ of subspaces of $X$, such that
	\begin{enumerate}
		\item \[X=\bigcup_{i=0}^n\bigcup_{U\in\cU^i}U;\]
		\item for every $i$ each $U\in\cU^i$ has diameter at most $f(r)$;
		\item  each $\cU^i$ is $r$-disjoint, that is, $d(V,U)\geq r$ for all $V,U\in\cU^i$ with $U\neq V$.
	\end{enumerate}	
	Any function $f\colon (0,\infty)\to \R$ with the above property is called a \emph{control function} for $X$.
	
	We call $\cU^0,\ldots,\cU^n$ an \emph{$r$-disjoint cover of diameter at most $f(r)$}.
\end{defi}
For a subspace $V$ of $X$ and $r\geq 0$ we define by $V^r$ the $r$-neighborhood of $V$, that is,
\[V^r:=\{x\in X\mid d(x,V)\leq r\}.\]
\begin{rem}
	\label{rem:lebesgue}
	Note that we can assume that the above covers also have a given Lebesgue number. This can be seen as follows. Let $k\geq0$ be fixed and take some $r+2k$-disjoint cover $\cU^0,\ldots,\cU^n$ of diameter at most $f(r+2k)$. Consider the cover given by $\cV^i:=\{U^k\mid U\in\cU^i\}$. Then $\cV^0,\ldots,\cV^n$ is an $r$-disjoint cover of diameter at most $f(r+2k)+2k$ and has Lebesgue number $k$, i.e.\ for every $x\in X$ there exists a $V\in\cV^i$ for some $i$ with $B_k(x)\subseteq V$.
\end{rem}

\begin{defi}
	Let $f\colon X\to Y$ be a function between metric spaces.
	\begin{enumerate}
		\item $f$ is \emph{uniformly expansive} if there exists a non-decreasing function \[\rho\colon [0,\infty)\rightarrow[0,\infty)\] such that for all $x,y \in X$
		\begin{equation*}
		d_Y(f(x), f(y))\leq \rho(d_X(x,y)).
		\end{equation*}		
		\item $f$ is \emph{effectively proper} if there exists a proper non-decreasing function \[\delta\colon [0,\infty)\rightarrow[0,\infty)\] such that for all $x,y \in X$
		\begin{equation*}
		\delta(d_X(x, y))\leq d_Y(f(x), f(y)).
		\end{equation*}		
		\item $f$ is a \emph{coarse embedding} if it is both uniformly expansive and effectively proper.
	\end{enumerate}
\end{defi}

From now on we fix a metric space $X$ of asymptotic dimension at most $n$ and a control function $f'$ with the additional property that $f'$ is non-decreasing. We define $f(x):=f'(3x)+3x$. Define $g\colon\N\to\R$ inductively by $g(0)=2$ and
\[g(k):=100f(g(k-1)).\]
Note that $g$ is non-decreasing and $g(k)\geq 2k$. Let $x_0\in X$ be a base-point.
\begin{lem}
	\label{lem:covers}
	There exist $\tfrac{9}{10}g(k)$-disjoint covers $\cU^0_k,\ldots, \cU^n_k$ of diameter at most $2f(g(k))$ such that
	\begin{enumerate}
		\item\label{prop:1} for every $j$, every $l<k$ and all $U\in \cU^j_k, V\in\cU^j_l$ we have $d(U,V)\leq l$ implies $V^l\subseteq U$;
		\item\label{prop:2} for $0\leq i\leq n$ with $i=k$ modulo $n+1$, we have $B_k(x_0)\subseteq U$ for some $U\in \cU_k^i$.
	\end{enumerate}
\end{lem}
Note that Property \eqref{prop:1} is the main condition. Property \eqref{prop:2} is only a technical point to ensure that the trees we will construct are connected. It ensures that for every $0\leq i\leq n$ eventually every point is contained in an element of $\cU^i$. It is only here that we use the existence of covers with large Lebesgue number from \cref{rem:lebesgue}.
\begin{proof}
%
For $k$, pick a $g(k)$-disjoint cover $\wt \cU_k^0,\ldots,\wt\cU_k^n$ of diameter at most $f(g(k))$ with Lebesgue number at least $k$. Let $0\leq i\leq n$ be such that $i=k$ modulo $n+1$. By renumbering $\wt \cU^0_k,\ldots, \wt\cU_k^n$, we can assume that $B_k(x_0)\subseteq U$ for some $U\in \wt\cU_k^i$. This is done to ensure Property~\eqref{prop:2}.

The cover $\cU_k^j$ will depend on $\wt\cU_k^j$ and the inductively defined $\cU_l^j$ for $l<k$. We begin with $\cU^j_0:=\wt \cU^j_0$. For $U\in\wt \cU_k^j$, we will define inductively $W_i(U)$ and $B_i(U)$ for $i\leq k$. First define $W_k(U):=U$. Assuming that $W_{i+1}(U)$ is defined, we define $B_{i}(U):=\{V\in \cU_{i}^j\mid d(V,W_{i+1}(U))<f(g(i))\}$ and
 \[W_i(U)=W_{i+1}(U)\cup\bigcup_{V\in B_i(U)}V^i.\]
Define $\cU^j_k:=\{W_0(U)\mid U\in\wt\cU^j_k\}$. 
By induction we will show that each element in $\cU^j_k$ has diameter at most $2f(g(k))$. This is clear for $\cU^j_0$ since $\cU^j_0=\wt \cU^j_0$ which has diameter at most $f(g(0))$ by assumption.

Since each $V\in B_i(U)$ has diameter at most $2f(g(i))$, the set $V^i$ has diameter at most $2f(g(i))+2i\leq 3f(g(i))$. Since also the distance from $V$ to $W_{i+1}(U)$ is at most $f(g(i))$, we have $W_i(U)\subseteq W_{i+1}(U)^{4f(g(i))}$. This implies that the diameter of $W_i(U)$ is at most $\diam W_{i+1}(U)+8f(g(i))$. As $9f(g(i))\leq f(g(i+1))$, inductively we see that 
\[\diam W_0(U)\leq \diam U+\sum_{i<k}8f(g(i))\leq \diam U+9f(g(k-1))\leq 2f(g(k)).\] Furthermore, it follows that the cover $\cU^j_k$ is $\tfrac{9}{10}g(k)$ disjoint.

We have to show that $\cU^j_k$ has the desired property. Suppose $l<k$, $U\in \wt\cU^j_k$, $V\in \cU^j_l$ with $d(W_0(U),V)\leq l$. 
As $W_{i-1}(U)\subseteq W_i(U)^{4f(g(i-1))}$, we have $d(V,W_i(U))\leq d(V,W_{i-1}(U))+4f(g(i-1))$. Similarly to the diameter estimate, it follows that 
\[d(W_{l}(U),V)\leq d(W_0(U),V)+\sum_{i<l}4f(g(i))< l+5f(g(l-1)).\] If $d(W_{l+1}(U),V)<l+5f(g(l-1))<f(g(l))$, then $V^l\subseteq W_l(U)\subseteq W_0(U)$ by construction. Otherwise, $d(V,V')<5f(g(l-1))+2l<\tfrac{9}{10}g(l)$ for some $V'\in B_{l}(U)$. But since $V,V'\in\cU_l^j$ and $\cU_l^j$ is $\tfrac{9}{10}g(l)$-disjoint, we have $V=V'$ and thus also $V^l\subseteq W_l(U)\subseteq W_0(U)$.
\end{proof}

Given covers $\cU^0_k,\ldots, \cU^n_k$ as in \cref{lem:covers}, we define trees $T^j$ as follows. The vertices of $T^j$ are the elements $\{(U,k)\mid k\in\N, U\in \cU^j_k\}$ and $(V,k),(U,k')\in T^j$ with $k<k'$ are connected by an edge if and only if $V\subseteq U$ and there is no $k<l<k'$ such that there exists $W\in\cU_l^j$ with $V\subseteq W$. Note that this in particular implies that $(V,k)$ and $(U,k')$ are connected by a sequence of edges if $V\subseteq U$. In the following we will denote the vertices $(U,k)$ only by $U$. By Property~\eqref{prop:1} of \cref{lem:covers}, the constructed graph contains no cycles. That it is connected can be seen as follows. Given vertices $U,U'$. Since $U$ and $U'$ are bounded, there is $k\in\N$ such that $U$ and $U'$ are both contained in $B_k(x_0)$. By enlarging $k$ we can assume $k=j\mod n+1$. Hence there exists $V\in\cU^j_k$ with $B_k(x_0)\subseteq V$ by Property \eqref{prop:2} of \cref{lem:covers}. It follows that both $U$ and $U'$ are connected to $V$. 

Every vertex $U\in \cU^j_k$ is only connected by an edge to one vertex $V$ with $U\subseteq V$ by definition of $T^j$. If $X$ is proper, then for every $l< k$ there can only be finitely many $V\in \cU^j_l$ that are contained in $U$ by disjointness of $\cU^j_l$. Hence the trees $T^j$ are locally finite if $X$ is proper.

For $x\in X$ define $\psi^j(x):=\min\{k\in\N\mid \exists U\in\cU_k^j, x\in U\}$ and let $\phi^j(x)\in \cU^j_{\psi^j(x)}$ be the element containing $x$. 
\begin{lem}
	The map $\phi^j\colon X\to T^j$ is uniformly expansive when $T^j$ is considered as a metric space with the $l^1$-path metric.
\end{lem}
\begin{proof}
	Given $x,y\in X$ with $d_{T^j}(\phi^j(x),\phi^j(y))\geq k\geq 4$, let $V\in \cU^j_m$ contain both $\phi^j(x)$ and $\phi^j(y)$ and let $m$ be minimal with this property. Without loss of generality we can assume $d_{T^j}(\phi^j(x),V)\geq \tfrac{k}{2}$. Then there are $W_x\in \cU^j_{l}, W_x'\in \cU^j_{l'}$ with $l'<l<m$, $\phi^j(x)\subseteq W_x'\subseteq W_x$, $\phi^j(y)\nsubseteq W_x$ and $d_{T^j}(\phi^j(x),W_x')\geq \tfrac{k}{2}-2$. These are the vertices in $T^j$ as depicted in the following figure:
	
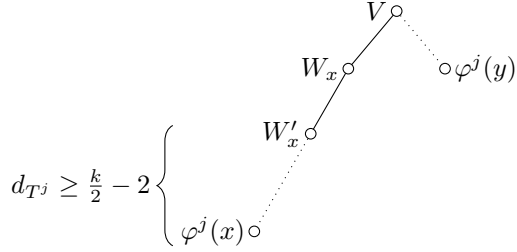
\begin{center}
			\begin{tikzpicture}
		\tikzstyle{every node}=[draw,circle,fill=white,minimum size=4pt,
		inner sep=0pt]
		\draw (0,0) node (1) [label=left:$V$] {}
		-- ++(230:1cm) node (2) [label=left:$W_x$] {}
		-- ++(240:1cm) node (3) [label=left:$W_x'$] {}
		++(240:1.5cm) node (4) [label=left:$\phi^j(x)$] {};
		\draw  (1)  ++(310:1cm) node (5) [label=right:$\phi^j(y)$] {};
		\draw [dotted] (3) -- (4);
		\draw [dotted] (1) -- (5);
		\draw [decorate,decoration={brace,amplitude=6pt},xshift=-4pt,yshift=0pt]
		(4) ++(190:1.1) -- ++(90:1.6)
		 node[midway,draw=none,anchor=east,label=left:$d_{T^j}\geq\tfrac{k}{2}-2~$]{};
		\end{tikzpicture}
	\end{center}
	 From $d_{T^j}(\phi^j(x),W_x')\geq \tfrac{k}{2}-2$ and $\phi^j(x)\subseteq W_x'$ it follows that $l'\geq \tfrac{k}{2}-2$. By Property \eqref{prop:1} from \cref{lem:covers}, the $l'$-neighborhood of $W_x'$ is contained in $W_x$ while, by definition of $\phi^j(Y)$, $y$ is not contained in $W_x$. Hence $d_X(x,y)\geq l'\geq \tfrac{k}{2}-2$.
	
	This shows $d_{T^j}(\phi^j(x),\phi^j(y))\leq 2d_X(x,y)+4$ and $\phi^j$ is coarsely (or large-scale) Lipschitz. In particular, it is uniformly expansive.
\end{proof}
We consider $\prod_{j=0}^nT^j$ with the metric $d((t_j)_j,(t_j')_j)=\sup_{0\leq j\leq n}d_{T^j}(t_j,t'_j)$. Now \cref{thm:main} is implied by the following theorem. 
\begin{thm}
	The map $\phi=\prod_{j=0}^n\phi^j\colon X\to \prod_{j=0}^nT^j$ is a coarse embedding.
\end{thm}
\begin{proof}
	We have already seen that $\phi$ is uniformly expansive. It remains to show that it is also effectively proper.
	
	Let $k\in\N$ and let $x,y\in X$ be such that $d(\phi(x),\phi(y))\leq k$. 
	For every $i\in\N$ there is some $0\leq j\leq n$ and $U\in\cU_i^j$ with $x\in U$. Therefore, there is some $j$ such that there exist numbers $i_1<\ldots<i_k<(n+1)k$ and elements $U_l\in \cU^j_{i_l}$ with $x\in U_l$ for $1\leq l\leq k$. By the construction from \cref{lem:covers}, we have $U_1\subseteq U_2\subseteq\ldots\subseteq U_k$, and, by the construction of $T^j$, we have $d_{T^j}(\phi^j(x),U_k)\geq k$. Therefore, $\phi^j(y)\subseteq U_{k}$ and $d_X(x,y)\leq \operatorname{diam} U_{k}\leq f(g(i_k))\leq f(g((n+1)k))$.
	
	Define $h\colon [0,\infty)\to [0,\infty)$ by $h(t):=\min\{i\in\N\mid t\leq f(g((n+1)k))\}$.
	Then $d(\phi(x),\phi(y))\geq h(d(x,y))$, so $\phi$ is effectively proper.
\end{proof}
\subsection*{Acknowledgments} I would like to thank David Rosenthal for helpful discussions and Ulrich Bunke, Fumiya Mikawa, Christoph Winges, Takamitsu Yamauchi and the referee for useful comments on a previous version.
\bibliographystyle{amsalpha}
\bibliography{trees}
\end{document}